\documentclass[12pt]{amsart} 

\usepackage{amsfonts,amsthm,amssymb,dsfont,stmaryrd,mathrsfs} 
\usepackage[leqno]{amsmath} 
\usepackage{array}
\usepackage{bm}
\usepackage{colortbl}
\usepackage[textwidth=16.45cm, textheight=21.9cm]{geometry} 
\usepackage{graphicx,caption} 
\usepackage[pdfstartview=FitH]{hyperref} 
\usepackage{pdfpages} 
\usepackage{latexsym}
\usepackage{longtable,supertabular, multirow} 
\usepackage{textcomp} 
\usepackage{ltxtable}
\usepackage{bibentry} \nobibliography* 
\usepackage{tikz, tikz-cd}
\usetikzlibrary{matrix,arrows,decorations.pathmorphing}

\usepackage{setspace}
\usepackage{mathtools} 

\theoremstyle{plain}
\newtheorem*{thm*}{Theorem} 
\newtheorem{thm}{Theorem}[section] 

\newtheorem{lem}[thm]{Lemma}
\newtheorem*{lem*}{Lemma}
\newtheorem{prop}[thm]{Proposition}
\newtheorem*{prop*}{Proposition}

\newtheorem{cor}[thm]{Corollary}
\newtheorem*{cor*}{Corollary}

\theoremstyle{definition}

\newtheorem*{defn*}{Definition}
\newtheorem{conj}[thm]{Conjecture}

\newtheorem*{exmp*}{Example}

\newtheorem*{remk*}{Remark}

\newtheorem*{prob*}{Problem}

\newtheorem*{ques*}{\frame{Question}}

\makeatletter
\def\blfootnote{\xdef\@thefnmark{}\@footnotetext} 
\makeatother

\def\nn{\mathbb{N}}
\def\zz{\mathbb{Z}}
\def\qq{\mathbb{Q}}

\def\cc{\mathbb{C}}

\def\ker{{\rm ker}}

\def\tr{{\rm tr}}
\def\id{{\rm id}}

\def\Aut{{\rm Aut}}

\def\ord{{\rm ord}}

\def\aug{{\rm aug}}

\begin{document}


\title[{MJD in $\zz[Q_8 \times C_p]$}]{The Multiplicative Jordan Decomposition in the Integral Group Ring $\zz[Q_8 \times C_p]$}

\date{\textsuperscript{\textcopyright}  2019. This manuscript version is made available under the CC-BY-NC-ND 4.0 license http://creativecommons.org/licenses/by-nc-nd/4.0}

\author{Wentang Kuo}
\address{Department of Pure Mathematics, Faculty of Mathematics, University of Waterloo, Waterloo, Ontario, N2L 3G1, Canada}
\email{wtkuo@uwaterloo.ca}

\author{Wei-Liang Sun}
\address{Department of Mathematics, National Taiwan Normal University, Taipei 11677, Taiwan, ROC}
\email{wlsun@ntnu.edu.tw}

\maketitle


\begin{abstract}
Let $p$ be a prime such that the multiplicative order $m$ of $2$ modulo $p$ is even. We prove that the integral group ring $\zz[Q_8 \times C_p]$ has the multiplicative Jordan decomposition property when $m$ is congruent to $2$ modulo $4$. There are infinitely many such primes and these primes include the case $p \equiv 3 \pmod{4}$. We also prove that $\zz[Q_8 \times C_5]$ has the multiplicative Jordan decomposition property in a new way.
\end{abstract}

\providecommand{\keywords}[1]{\textit{Keywords:} #1}


\section{Introduction}
\label{sec1}

For a finite dimensional algebra $A$ over a perfect field $F$, every element $\alpha$ of $A$ has a unique additive Jordan decomposition $\alpha = \alpha_s + \alpha_n$, where $\alpha_s$ is a semisimple element, $\alpha_n$ is a nilpotent element and $\alpha_s \alpha_n = \alpha_n \alpha_s$. Here, an element is called semisimple if its minimal polynomial over $F$ has no repeated roots in the algebraic closure of $F$. If $\alpha$ is a unit, then $\alpha_s$ is also invertible and $\alpha_u = 1 + \alpha_s^{-1} \alpha_n$ a unipotent unit with $\alpha_s \alpha_u = \alpha_u \alpha_s$. Then $\alpha = \alpha_s \alpha_u$ is the unique multiplicative Jordan decomposition of $\alpha$. Let $F$ be the field $\qq$ of rational numbers and $A = \qq[G]$ be the rational group algebra of a finite group $G$ over $\qq$. Viewing $\zz[G]$ as a subring of $\qq[G]$, each unit $\alpha$ of $\zz[G]$ has the unique multiplicative Jordan decomposition $\alpha = \alpha_s \alpha_u$ in $\qq[G]$. Usually, $\alpha_s$ and $\alpha_u$ do not always lie in $\zz[G]$. Following \cite{AroHalPas} and \cite{HalPasWil}, we say that $\zz[G]$ has the multiplicative Jordan decomposition property (MJD) if for every unit $\alpha$ of $\zz[G]$, those $\alpha_s$ and $\alpha_u$ are contained in $\zz[G]$.

The MJD problem on group rings $R[G]$ of finite groups over integral domains $R$ was first proposed by A.W. Hales and I.B.S. Passi \cite[Concluding remarks]{HalPas} in 1991. It is a particularly interesting problem of classifying those finite groups $G$ for which $\zz[G]$ has MJD. In 2007, Hales, Passi and L.E. Wilson \cite[Theorem~29]{HalPasWil} (together with \cite{HalPasWil2}) shows that if $\zz[G]$ has MJD, then one of the following holds: 
\begin{itemize}
\item[(i)]
$G$ is either abelian or of the form $Q_8 \times E \times A$ where $Q_8$ is the quaternion group of order eight, $E$ is an elementary abelian $2$-group and $A$ is abelian of odd order such that the multiplicative order of $2$ modulo $|A|$ is odd (in this case, $\qq[G]$ contains no nonzero nilpotent elements);
\item[(ii)]
$G$ has order $2^a 3^b$ for some nonnegative integers $a, b$;
\item[(iii)]
$G = Q_8 \times C_p$ for some prime $p \geq 5$ such that the multiplicative order of $2$ modulo $p$ is even;
\item[(iv)]
$G$ is the split extension of $C_p$ ($p \geq 5$) by a cyclic group $\langle g \rangle$ of order $2^k$ or $3^k$ for some $k \geq 1$, and $g^2$ or $g^3$ acts trivially on $C_p$.
\end{itemize}
This theorem is a one-way result: there are some groups $G$ above so that $\zz[G]$ does not have MJD. It should be classified which group rings have MJD for groups in each case. This classification has not been completed and there are some current results for each case. For any group in case (i), $\zz[G]$ has MJD since it contains no nonzero nilpotent elements (see Theorem~\ref{thmnilp}). For case (ii), the classification has been completed by S.R. Arora, Hales, Passi, Wilson \cite{AroHalPas, HalPasWil, HalPasWil2}, M.M. Parmenter \cite{Par}, C.-H. Liu and D.S. Passman \cite{LiuPas, LiuPas4, LiuPas2}. For case (iv), it is still incomplete but many of possibilities are classified by several authors as above in \cite{AroHalPas, LiuPas3, HalPas2}. On the contrary, no progress for case (iii) has been made except $p = 5$ by X.-L. Wang and Q.-X. Zhou \cite{WanZho}. The survey paper \cite{HalPas2} gives the progress on the MJD problem and related topics.

The aim of this paper is to prove that there are infinitely many primes $p$ in the case (iii) such that $\zz[Q_8 \times C_p]$ has MJD. This gives a great advance to case (iii) since 1991. More precisely, we will prove the following result. 
\begin{thm}
\label{thm2}
If the multiplicative order of $2$ modulo $p$ is congruent to $2$ modulo $4$, then $\zz[Q_8 \times C_p]$ has MJD.
\end{thm}

It is interesting that for groups in cases (ii) and (iv), their integral group rings usually do not have MJD when group orders are large enough (except for orders $2p$ and $4p$ in case (iv)).

Note that if the multiplicative order of $2$ modulo $p$ is even, says $2 k$, then $k$ divides $(p-1)/2$ since $2 k$ divides $p-1$. Thus, $k$ is odd when $p \equiv 3 \pmod{4}$. We have the following immediate consequence. 

\begin{cor}
If the multiplicative order of $2$ modulo $p$ is even and $p \equiv 3 \pmod{4}$, then $\zz[Q_8 \times C_p]$ has MJD.
\end{cor}

Assume that {\it the multiplicative order of $2$ modulo an odd prime $p$ is even} so that the group ring $\qq[Q_8 \times C_p]$ has nonzero nilpotent elements. In Section~\ref{sec2}, we first study nonzero nilpotent elements of $\qq[G]$. Each nonzero nilpotent element has a clear form. Then we find the Jordan decomposition for each non-semisimple unit of $\zz[Q_8 \times C_p]$ in Section~\ref{sec3}. Studying semisimple parts of non-semisimple units will lead us to prove Theorem~\ref{thm2} in Section~\ref{sec4}. Back to the prime $p$. When $p \equiv 3 \pmod{4}$, we already have the positive result by the previous corollary. For $p \equiv 1 \pmod{4}$, we remark that $281$ is the smallest prime satisfying the condition of Theorem~\ref{thm2}. Note that it was proved that $\zz[Q_8 \times C_5]$ has MJD in \cite{WanZho}. They proved this by computing a certain subgroup of $\mathcal{U}(\zz[\varepsilon_5])$ for some primitive $5$-th root $\varepsilon_5$ of $1$ in $\cc$. However, it is not easy to compute subgroups of $\mathcal{U}(\zz[\varepsilon_p])$ for large $p$. In Section~\ref{sec5}, we will provide a new proof for $p = 5$ without computing subgroups of $\mathcal{U}(\zz[\varepsilon_5])$. Finally, we will compute the Chebotarev density for primes satisfying the condition of Theorem~\ref{thm2} in Section~\ref{sec6}. Thus we show that there are infinitely many such primes $p$ so that $\zz[Q_8 \times C_p]$ has MJD.

\section{Nilpotent Elements in $\qq[Q_8 \times C_p]$}
\label{sec2}

For a group $G$ and $\alpha = \sum_{g \in G} \alpha_g g \in \qq[G]$ with $\alpha_g \in \qq$, we denote by $\aug(\alpha) = \sum_{g \in G} \alpha_g$ the augmentation of $\alpha$. If $\aug(\alpha) = 1$, $\alpha$ is said to have augmentation $1$. View $\zz[G]$ as a subring of $\qq[G]$. Denote $\mathcal{U}(\zz[G])$ the unit group of $\zz[G]$ and $\mathcal{U}_1(\zz[G])$ the group of units in $\mathcal{U}(\zz[G])$ with augmentation $1$. 

Let $p$ be an odd prime such that the multiplicative order of $2$ modulo $p$ is even. Let $G = Q_8 \times C_p$ where $$Q_8 = \langle a, b \mid a^4 = 1, \; a^2 = b^2, \; b a b^{-1} = a^{-1} \rangle$$ and $$C_p = \langle t \mid t^p = 1 \rangle.$$ Denote $c = a b$ and $z = a^2 = b^2 = c^2$. In this section, we will study nilpotent elements in $\qq[G]$. It will be used when we try to find the Jordan decomposition of a non-semisimple unit in the next section.

Since the multiplicative order of $2$ modulo $p$ is even, there exist $r, s \in \zz[\varepsilon]$ such that $$r^2 + s^2 = -1$$ where $\varepsilon$ is a primitive $p$-th root of $1$ in $\cc$ (\cite[p. 153]{GiaSeh}). 
Then we have a ring homomorphism $$\rho: \qq[G] \to M_2(\qq(\varepsilon))$$ defined by $$a \mapsto \left( \begin{array}{cc} 0 & 1 \\ -1 & 0 \end{array} \right), \quad b \mapsto \left( \begin{array}{cc} r & s \\ s & -r \end{array} \right), \quad t \mapsto \left( \begin{array}{cc} \varepsilon & 0 \\ 0 & \varepsilon \end{array} \right).$$ Clearly, $\rho(\zz[G]) \subseteq M_2(\zz[\varepsilon])$. We also consider the following two ring homomorphisms $$\phi: \qq[G] \to \qq[G/G'] = \qq[G/ \langle z \rangle]$$ defined by $a \mapsto \overline{a}$, $b \mapsto \overline{b}$ and $t \mapsto t$ with $\ker(\phi) = (1 - z) \qq[G]$ and $$\theta: \qq[G] \to \qq[G/C_p] = \qq[Q_8]$$ defined by $a \mapsto a$, $b \mapsto b$ and $t \mapsto 1$ with $\ker(\theta) = (1 - t) \qq[G]$.

Note that the center of $G$ is $\langle z \rangle \times \langle t \rangle$ and we have the following immediate consequence which will be used in Section \ref{sec3}.

\begin{lem}
\label{lemJan1}
Let $g \in G$.
\begin{enumerate}
\item[(i)]
If $g$ is not central, then $\tr(\rho(g)) = 0$.
\item[(ii)]
If $g = z^i t^k$, 
then $\tr(\rho(g)) = (-1)^i 2 \varepsilon^k$.
\item[(iii)]
$\tr(\rho(\zz[G])) \subseteq 2 \zz[\varepsilon]$.
\end{enumerate}
\end{lem}

Now, consider the following ring isomorphism $$\Delta: \qq[C_p] \overset{\simeq}{\to} \qq \oplus \qq(\varepsilon)$$ by sending $t$ to $(1, \varepsilon)$. For an element $\alpha(t) \in \qq[C_p]$, we denote $$\Delta(\alpha(t)) = (\alpha(1), \alpha(\varepsilon))$$ for convenience. Note that $\alpha(1) = \aug(\alpha)$. Under the isomorphism, if $\alpha(1) = 0$ and $\alpha(\varepsilon) = 0$, we can conclude that $\alpha(t) = 0$. Next, we consider a certain situation under $\rho$.

\begin{lem}
\label{lemJan3}
Let $\alpha = \sum_{g \in Q_8} \alpha_g(t) g \in \qq[G]$ for some $\alpha_g(t) \in \qq[C_p]$. Let $\xi_g = \alpha_g(\varepsilon) - \alpha_{g z}(\varepsilon)$ for $g \in \{1, a, b, c\}$. Then $\rho(\alpha)$ is nilpotent if and only if $\xi_1 = 0$ and $\xi_a^2 + \xi_b^2 + \xi_c^2 = 0$.
\end{lem}

\begin{proof}
Observe that $$\rho(\alpha) = \sum_{g \in Q_8} \alpha_g(\varepsilon) \rho(g) = \left( \begin{array}{rr} \xi_1 + r \xi_b + s \xi_c & \xi_a + s \xi_b - r \xi_c \\ - \xi_a + s \xi_b - r \xi_c & \xi_1 - r \xi_b - s \xi_c \end{array}\right).$$ Then $\tr(\rho(\alpha)) = 2 \xi_1$ and $$\det(\rho(\alpha)) = \xi_1^2 + \xi_a^2 + \xi_b^2 + \xi_c^2$$ since $r^2 + s^2 = -1$. The result holds since $X^2 - \tr(\rho(\alpha)) X + \det(\rho(\alpha))$ is the characteristic polynomial of $\rho(\alpha)$. 
\end{proof}

We need the following well-known result.

\begin{thm}[{\cite[p. 172]{Seh} or \cite[Theorem 7.4.11]{PMSeh}}]
\label{thmnilp}
Let $H$ be a finite group of order $2^k m$ with odd $m$. Then $\qq[H]$ has no nonzero nilpotent elements if and only if $H$ is either abelian or $H = Q_8 \times E \times A$ where $E$ is an elementary abelian $2$-group and $A$ is an abelian group of order $m$ such that the multiplicative order of $2$ modulo $m$ is odd.
\end{thm}

Now, each nilpotent element in $\qq[G]$ has a nice form.

\begin{prop}
\label{prop2}
Let $\alpha = \sum_{g \in Q_8} \alpha_g(t) g \in \qq[G]$ with $\alpha_g(t) \in \qq[C_p]$. Then $\alpha$ is nilpotent if and only if $$\alpha = \sum_{g \in \{a, b, c\}} \alpha_g(t) g (1 - z)$$ with $$\alpha_a(t)^2 + \alpha_b(t)^2 + \alpha_c(t)^2 = 0.$$ In this case, $\alpha^2 = 0$ and $\alpha_g(1) = 0$ for $g \in \{a, b, c\}$.
\end{prop}

\begin{proof}
Assume that $\alpha$ is nilpotent and so is $\phi(\alpha)$. Since $G / \langle z \rangle = (Q_8 / \langle z \rangle) \times C_p \simeq C_2 \times C_2 \times C_p$ is abelian, its rational group ring has no nonzero nilpotent elements by Theorem~\ref{thmnilp}. We see that $\phi(\alpha) = 0$ and we get $\sum_{g \in \{1, a, b, c\}} (\alpha_g(t) + \alpha_{g z}(t)) \overline{g} = 0$ in $\qq[G / \langle z \rangle]$. Hence, $\alpha_{g z}(t) = - \alpha_g(t)$ for $g \in \{1, a, b, c\}$. 

Note that $\theta(\alpha)$ is also nilpotent in $\qq[Q_8]$. By Theorem~\ref{thmnilp}, $\theta(\alpha) = 0$ and it implies that $\sum_{g \in Q_8} \alpha_g(1) g = 0$. Hence, $\alpha_g(1) = 0$ for all $g \in Q_8$. Since $\rho(\alpha)$ is nilpotent, we get $\alpha_1(\varepsilon) - \alpha_{z}(\varepsilon) = 0$ by Lemma~\ref{lemJan3}. This implies that $\alpha_1(\varepsilon) = \alpha_z(\varepsilon) = 0$ since $\alpha_z(t) = - \alpha_1(t)$. Now we have $\alpha_1(1) = 0$, $\alpha_1(\varepsilon) = 0$ and hence $\alpha_1(t) = 0$. Moreover, $\alpha_z(t) = 0$. We obtain $\alpha = \sum_{g \in \{a, b, c\}} \alpha_g(t) g (1 - z)$ and $\aug(\alpha_g(t)) = \alpha_g(1) = 0$ for $g \in \{a, b, c\}$.

By Lemma~\ref{lemJan3} again, we have $\sum_{g \in \{a, b, c\}} (\alpha_g(\varepsilon) - \alpha_{g z}(\varepsilon))^2 = 0$. Note that $\alpha_{g z}(t) = - \alpha_g(t)$ for $g \in \{a, b, c\}$ and we obtain $\sum_{g \in \{a, b, c\}} \alpha_g(\varepsilon)^2 = 0$. Since $\sum_{g \in \{a, b, c\}} \alpha_g(1)^2 = 0$, it follows that $\sum_{g \in \{a, b, c\}} \alpha_g(t)^2 = 0$. 

Conversely, if $\alpha = \sum_{g \in \{a, b, c\}} \alpha_g(t) g (1 - z)$, then $\alpha^2 = - 2 (\alpha_a(t)^2 + \alpha_b(t)^2 + \alpha_c(t)^2) (1 - z)$. The result follows when $\sum_{g \in \{a, b, c\}} \alpha_g(t)^2 = 0$.
\end{proof}

\section{Semisimple and Nilpotent Parts}
\label{sec3}

In this section, we are going to give explicit forms of semisimple and nilpotent parts of a non-semisimple unit in $\zz[G]$ where $G = Q_8 \times C_p$. Recall that $C_p = \langle t \mid t^p = 1 \rangle$, $$Q_8 = \langle a, b \mid a^4 = 1, \; a^2 = b^2, \; b a b^{-1} = a^{-1} \rangle$$ and we denote $c = a b$ and $z = a^2 = b^2 = c^2$. First of all, we have the following observation.

\begin{lem}
\label{lem8}
If $u \in \mathcal{U}_1(\zz[G])$ is non-semisimple, then $u_s$ is central.
\end{lem}

\begin{proof}
Let $u \in \mathcal{U}_1(\zz[G])$ be non-semisimple. Consider the Wedderburn decomposition of $\qq[G]$ that $$\qq[G] \simeq (\qq[C_p]) [Q_8] \simeq \qq[Q_8] \oplus \qq(\varepsilon) [Q_8] \simeq \qq[Q_8] \oplus 4 \qq(\varepsilon) \oplus M_2(\qq(\varepsilon)).$$ Here, $4 \qq(\varepsilon)$ is the direct sum of four copies of $\qq(\varepsilon)$. The two maps $\theta: \qq[G] \to \qq[Q_8]$ and $\rho : \qq[G] \to M_2(\qq[\varepsilon])$ given in Section~\ref{sec2} can be viewed as two projections of this decomposition. It suffices to show that both $\theta(u_s)$ and $\rho(u_s)$ are central. 

Since $u$ is non-semisimple, $\rho(u)$ is a non-semisimple unit in $M_2(\qq[\varepsilon])$. Thus, $\rho(u)$ has repeated eigenvalues. It follows that $\rho(u_s)$ is a diagonal matrix since its minimal polynomial has degree $1$. In particular, $\rho(u_s)$ is central. 

For $\theta(u_s)$, we observe that $\theta(u_s) = \theta(u)$ since $\theta(u_n) = 0$ by Theorem~\ref{thmnilp}. Moreover, $\zz[Q_8]$ has only trivial units by Higman's theorem (see \cite[Proposition~(2.2)]{Seh2}). So $\theta(u) = g z^i$ for some $g \in \{1, a, b, c\}$ and $i \in \{0, 1\}$ since $\aug(\theta(u)) = \aug(u) = 1$. Note that $g z^i = \theta(g z^i)$ and $\ker(\theta|_{\zz[G]}) = (1-t) \zz[G]$. Then $u = g z^i + (1 - t) \beta$ for some $\beta \in \zz[G]$. We will see that $g = 1$ and then $\theta(u)$ is central. 

Note that $\tr(\rho(u)) = (-1)^i \tr(\rho(g)) + (1 - \varepsilon) \tr(\rho(\beta))$ and we deduce that $\tr(\rho(u)) \in (-1)^i \cdot 2 + (1 - \varepsilon) 2 \zz[\varepsilon]$ if $g = 1$; and $\tr(\rho(u)) \in (1 - \varepsilon) 2 \zz[\varepsilon]$ if $g \in \{a, b, c\}$ by Lemma~\ref{lemJan1}. Let $\lambda$ be the repeated eigenvalue of $\rho(u)$. Then $\lambda = \tr(\rho(u)) / 2 \in \zz[\varepsilon]$. Moreover, it is a unit in $\zz[\varepsilon]$ because $\rho(u_s) = \lambda I$ is invertible and $\lambda^{-1} = \tr(\rho(u^{-1})) / 2 \in \zz[\varepsilon]$. Consider the norm map $N = N_{\qq(\varepsilon) / \qq}$, namely, $N(\alpha) = \prod_{\sigma \in \Aut(\qq(\varepsilon) / \qq)} \sigma(\alpha)$ for $\alpha \in \qq(\varepsilon)$. If $g \in \{a, b, c\}$, then $\tr(\rho(g)) = 0$ and $\lambda \in (1 - \varepsilon) \zz[\varepsilon]$. It follows that $N(\lambda)$ is divided by $N(1 - \varepsilon) = p$ in $\zz$. It is impossible because $\lambda$ is a unit and $N(\lambda) = \pm 1$. So we get $g = 1$.
\end{proof}

The above lemma is false when $u$ is not a unit. For instance, let $$u = \sum_{g \in \{a, b, c\}} [1 + t + \cdots + t^{p-1} + \alpha_g(t)] g - \alpha_g(t) g z$$ where $\alpha_g(t) \in \zz[C_p]$ not all zero and $\alpha_a(t)^2 + \alpha_b(t)^2 + \alpha_c(t)^2 = 0$. Such $\alpha_a(t)$, $\alpha_b(t)$ and $\alpha_c(t)$ exist since the multiplicative order of $2$ modulo $p$ is even (see \cite[p. 153]{GiaSeh}). Here, $u$ is not a unit since $\aug(u) = 3 p \neq \pm 1$. Moreover, $u_n = \sum_{g \in \{a, b, c\}} \alpha_g(t) g (1 - z) \neq 0$ and $u$ is non-semisimple. In particular, $u_s = (1 + t + \cdots + t^{p-1}) (a + b + c)$ is not central.

Recall the isomorphism $\Delta: \qq[C_p] \to \qq \oplus \qq(\varepsilon)$ given by $\Delta(t) = (1, \varepsilon)$ and we have the embedding $$\zz[C_p] \hookrightarrow \zz \oplus \zz[\varepsilon].$$ Restricting on units of augmentation $1$, we have the following well-known result.

\begin{prop}[{\cite[Proposition (10.7)]{Seh2}}]
\label{prop3}
Let $$\mathcal{U}_1(\zz[\varepsilon])= \{ y \in \mathcal{U}(\zz[\varepsilon]) \mid y \equiv 1 \pmod{1 - \varepsilon}\},$$ then the map $$\delta: \mathcal{U}_1(\zz[C_p]) \to \mathcal{U}_1(\zz[\varepsilon])$$ given by $t \mapsto \varepsilon$ is a group isomorphism.
\end{prop}

Now, for a non-semisimple unit in $\zz[G]$, we can know what its semisimple and nilpotent parts in $\qq[G]$ look like. The following theorem is essentially due to \cite[Theorem~2.1]{WanZho} but their proof looks complicated. Here, we rewrite the statement in a clearer way and prove the result using Lemma~\ref{lem8} and Proposition~\ref{prop2}.

\begin{thm}
\label{thm1}
Let $u \in \mathcal{U}_1(\zz[G])$ and write $u = \sum_{g \in Q_8} f_g(t) g$. Then $u$ is non-semisimple if and only if one of $f_g(t) - f_{gz}(t) \neq 0$ for $g \in \{a, b, c\}$ and $$\sum_{g \in \{a, b, c\}} (f_g(t) - f_{g z}(t))^2 = 0.$$ In this situation, we have the followings.
\begin{itemize}
\item[(i)]
The semisimple and nilpotent parts are $$u_s = f_1(t) + f_z(t) z + \frac{1}{2} \sum_{g \in \{a, b, c\}} (f_g(t) + f_{g z}(t)) g (1 + z)$$ and $$u_n = \frac{1}{2} \sum_{g \in \{a, b, c\}} (f_g(t) - f_{g z}(t)) g (1 - z)$$ with $u_n^2 = 0$.
\item[(ii)]
$(f_1(1), f_z(1)) \in \{(1, 0), (0,1)\}$ and $f_g(1) = 0$ for $g \in Q_8 \setminus \{1, z\}$.
\item[(iii)]
$f_1(t) - f_z(t) \in \mathcal{U}(\zz[C_p])$.
\end{itemize}
\end{thm}

\begin{proof}
Assume that $u$ is non-semisimple. By Lemma~\ref{lem8}, $u_s$ is central. Since conjugacy classes of $G$ are $\{t^i\}$, $\{z t^i\}$, $\{g t^i, g z t^i\}$ for $0 \leq i \leq p-1$ and $g \in \{a, b, c\}$, it follows that $u_s = d_1(t) + d_z(t) z + \sum_{g \in \{a, b, c\}} d_g(t) g (1 + z)$ for some $d_g(t) \in \qq[C_p]$. Then $$u_n = (f_1(t) - d_1(t)) + (f_z(t) - d_z(t)) z + \sum_{g \in \{a, b, c\}} (f_g(t) - d_g(t)) g + (f_{gz}(t) - d_g(t)) g z.$$ By Proposition~\ref{prop2}, we have 
\begin{enumerate}
\item[(a)]
$f_1(t) - d_1(t) = f_z(t) - d_z(t) = 0$,
\item[(b)]
$f_{gz}(t) - d_g(t) = - (f_{g}(t) - d_g(t))$ for $g \in \{a, b, c\}$, and
\item[(c)]
$(f_{a}(t) - d_a(t))^2 + (f_{b}(t) - d_b(t))^2 + (f_{c}(t) - d_c(t))^2 = 0$.
\end{enumerate}
It follows that $d_1(t) = f_1(t)$, $d_z(t) = f_z(t)$ and $d_g(t) = \frac{1}{2} (f_g(t) + f_{g z}(t))$ for $g \in \{a, b, c\}$. Clearly, $u_n^2 = 0$ from the Wedderburn decomposition of $\qq[G]$. We obtain (i). Moreover, by (c), we have $$(f_{a}(t) - f_{a z}(t))^2 + (f_{b}(t) - f_{b z}(t))^2 + (f_{c}(t) - f_{c z}(t))^2 = 0$$ via $f_g(t) - d_g(t) = (f_g(t) - f_{g z}(t))/2$ for $g \in \{a, b, c\}$.

By Proposition~\ref{prop2} again, we have $f_{g}(1) - d_g(1) = 0$ for $g \in \{a, b, c\}$. Note that $\theta(u_s) = z^i$ for some $i \in \{0, 1\}$ since $\theta(u_s) = \theta(u)$ and it is a central unit in $\zz[Q_8]$. It follows that $(d_1(1), d_z(1)) = (1, 0)$ if $i = 0$ and $(d_1(1), d_z(1)) = (0, 1)$ if $i = 1$, and $d_g(1) = 0$ for $g \neq 1, z$. Thus, we have $(f_1(1), f_z(1)) \in \{(1, 0), (0,1)\}$, $f_g(1) = 0$ and $f_{g z}(1) = 2 d_g(1) - f_g(1) = 0$ for $g \in \{a, b, c\}$. This gives (ii). 

For (iii), let $w(t) = f_1(t) - f_z(t)$ and we first observe from (i) that $\rho(u_s) = \lambda I$ where $\lambda = w(\varepsilon)$. Note that $\theta(u_s) = z^i$ implies that $w(1) = (-1)^i$. As in the proof of Lemma~\ref{lem8}, we have $\lambda \in \mathcal{U}(\zz[\varepsilon])$ and $\lambda = \tr(\rho(u)) / 2 \in (-1)^i + (1 - \varepsilon) \zz[\varepsilon]$. It follows that $\Delta((-1)^i w(t)) = (1, (-1)^i w(\varepsilon))$ and $(-1)^i w(\varepsilon) \in \mathcal{U}_1(\zz[\varepsilon])$. By Proposition~\ref{prop3}, we can conclude that $(-1)^i w(t) \in \mathcal{U}_1(\zz[C_p])$. Hence, $f_1(t) - f_z(t) \in \mathcal{U}(\zz[C_p])$.

Conversely, assume that $u \in \mathcal{U}_1(\zz[G])$ and $\sum_{g \in \{a, b, c\}} \beta_g^2 = 0$ with some $\beta_g \neq 0$ where $\beta_g = f_g(t) - f_{gz}(t)$ for $g \in \{a, b, c\}$. Then we set $$\alpha = \frac{1}{2} \sum_{g \in \{a, b, c\}} \beta_g g (1 - z) \neq 0.$$ Observe that $\alpha^2 = -\frac{1}{2} (\beta_a^2 + \beta_b^2 + \beta_c^2) (1 - z) = 0$ and $\alpha$ is nilpotent. Moreover, it is easy to see that $u - \alpha$ is central since it is a $\qq$-linear combination of class sums. It follows that $u - \alpha$ is semisimple and $u = (u - \alpha) + \alpha$ is the Jordan decomposition of $u$. In particular, $u$ is non-semisimple.
\end{proof}

By Theorem~\ref{thm1}(i), $\zz[G]$ has MJD if and only if for every non-semisimple unit $u = \sum f_g(t) g$, we have $f_g(t) + f_{g z}(t) \in 2 \zz[C_p]$ for each $g \in \{a, b, c\}$. We will focus on each $f_g(t) + f_{g z}(t)$ when a non-semisimple unit $u$ is given. For convenience, we may multiply $u$ by any central unit $w$ by the following obvious observation.

\begin{lem}
\label{lem9}
Let $H$ be a finite group and $u, w \in \mathcal{U}(\zz[H])$. If $w$ is central, then $(u w)_s = u_s w$ and $(u w)_n = u_n w$. In particular, $u$ is non-semisimple if and only if $u w$ is non-semisimple, and $u_s \in \zz[H]$ if and only if $(u w)_s \in \zz[H]$.
\end{lem}

Let $u \in \mathcal{U}_1(\zz[G])$ be non-semisimple and write $u = \sum_{g \in Q_8} f_g(t) g$. By Theorem~\ref{thm1}(iii), $f_1(t) - f_z(t)$ is a unit in $\zz[C_p]$. Note that it is also a central unit in $\zz[G]$. By Theorem~\ref{thm1}(ii), either $(f_{1}(1), f_z(1)) = (1, 0)$ or $(f_{1}(1), f_z(1)) = (0, 1)$. If $f_1(1) = 1$, then let $w = (f_1(t) - f_z(t))^{-1}$; if $f_z(1) = 1$, then let $w = - z (f_1(t) - f_z(t))^{-1}$. Thus, $w$ is a central unit of augmentation $1$ in $\zz[G]$ and $u w \in \mathcal{U}_1(\zz[G])$. By Lemma~\ref{lem9}, $u_s \in \zz[G]$ if and only if $(u w)_s \in \zz[G]$. Write $u w = \sum_{g \in Q_8} h_g(t) g$ and we have $h_1(t) - h_z(t) = 1$. We now consider that $$\mathcal{V} = \left\{v \in \mathcal{U}_1(\zz[G]) \mid v_n \neq 0 \text{ and $h_1(t) - h_z(t) = 1$ where } v = \sum_{g \in Q_8} h_g(t) g \right\}.$$ Then we have proved the following result.

\begin{prop}
\label{prop4}
For any non-semisimple $u \in \mathcal{U}_1(\zz[G])$, there exists a central $w \in \mathcal{U}_1(\zz[G])$ such that $u w \in \mathcal{V}$. Moreover, if $v_s \in \zz[G]$ for each $v \in \mathcal{V}$, then $\zz[G]$ has MJD.
\end{prop}

From now on, we will focus on non-semisimple units in $\mathcal{V}$. Before we end this section, we give the following result which will be used in Section~\ref{sec5}.

\begin{lem}
\label{lem7}
Let $u \in \mathcal{V}$. Then $u^{-1} = u_s^{-1} - u_n$ is the Jordan decomposition of $u^{-1}$. In particular, $u_s + u_s^{-1} \in \zz[G]$. 
\end{lem}

\begin{proof}
Let $u \in \mathcal{V}$ and write $u = \sum_{g \in Q_8} f_g(t) g$. By Theorem~\ref{thm1}(i), we can write $$u_s = (f_1(t) - f_z(t)) \frac{1-z}{2} + \sum_{g \in \{1, a, b, c\}} (f_g(t) + f_{g z}(t)) g \frac{1+z}{2}$$ and $u_n \in \zz[G] (\frac{1-z}{2})$. Thus, we have $u_s u_n = u_n u_s = u_n$ since $f_1(t) - f_z(t) = 1$ and $\frac{1-z}{2}$, $\frac{1+z}{2}$ are primitive central idempotents in $\qq[G]$. Multiplying $u_s^{-1}$ on both sides, we obtain $u_n = u_s^{-1} u_n = u_n u_s^{-1}$. It follows that $(u_s + u_n) (u_s^{-1} - u_n) = 1$ and $(u_s^{-1} - u_n) (u_s + u_n) = 1$ since $u_n^2 = 0$. Hence, $u^{-1} = u_s^{-1} - u_n$ and it is also the Jordan decomposition of $u^{-1}$ in $\qq[G]$. Finally, $u_s + u_s^{-1} = u + u^{-1} \in \zz[G]$. 
\end{proof}

\section{$*$-invariant and the Proof of Theorem \ref{thm2}}
\label{sec4}

Let $*$ be the involution of $\qq[G]$ induced by $g \mapsto g^{-1}$ for group elements $g \in G$. An element $\alpha$ in $\qq[G]$ is called $*$-invariant if $\alpha^* = \alpha$. Let $\mathcal{V}$ be defined as above Proposition~\ref{prop4} and let $u \in \mathcal{V}$. Then $u$ is non-semisimple and $f_1(t) - f_z(t) = 1$ when we write $u = \sum_{g \in Q_8} f_g(t) g$. Moreover, we have $f_1(1) = 1$, $f_g(1) = 0$ for $g \neq 1$ by Theorem~\ref{thm1}(ii). For convenience, we set $$F_g = f_g(t) + f_{g z}(t) \quad \text{for } g \in \{1, a, b, c\}.$$ As we mentioned before, if we can show that $F_g \in 2 \zz[C_p]$ for $g \in \{a, b, c\}$, then $u_s \in \zz[G]$ by Theorem~\ref{thm1}(i) and $\zz[G]$ will have MJD. In this section, we will prove that $F_g^* = F_g$ for each $g$. Once we have this, then we can prove Theorem~\ref{thm2}.

Recall the homomorphism $\phi : \qq[G] \to \qq[G/G']$ by sending $a \mapsto \overline{a}$, $b \mapsto \overline{b}$ and $t \mapsto t$, and $G / G' = \langle \overline{a} \rangle \times \langle \overline{b} \rangle \times \langle t \rangle = C_2 \times C_2 \times C_p$. Then $$\phi(u) = F_1 + F_a \overline{a} + F_b \overline{b} + F_c \overline{c}$$ is a unit in $\zz[C_2 \times C_2 \times C_p]$. Note also that $\qq[C_2 \times C_2 \times C_p] \simeq 4 \qq[C_p]$ and this isomorphism can be obtained by sending $$\overline{a} \mapsto (1, 1, -1, -1), \quad \overline{b} \mapsto (1, -1, 1, -1), \quad \text{and} \quad t \mapsto (t, t, t, t).$$ Via this isomorphism, we have $\zz[C_2 \times C_2 \times C_p] \hookrightarrow 4 \zz[C_p]$ and we obtain four units in $\zz[C_p]$, namely $$\phi(u) \mapsto (w_1, w_a, w_b, w_c)$$ where \begin{align*} \left\{\begin{array}{ccl}w_1 & = & F_1 + F_a + F_b + F_c, \\ w_a & = & F_1 + F_a - F_b - F_c, \\ w_b & = & F_1 - F_a + F_b - F_c, \\ w_c & = & F_1 - F_a - F_b + F_c. \end{array} \right. \tag{EQ} \label{sec4eq1}\end{align*}  
Since $f_1(1) = 1$ and $f_g(1) = 0$ for $g \in Q_8 \setminus \{1\}$, it follows that $\aug(F_1) = 1$ and $\aug(F_a) = \aug(F_b) = \aug(F_c) = 0$. Thus, $$w_1, w_a, w_b, w_c \in \mathcal{U}_1(\zz[C_p]).$$ In fact, these four units were used to prove the case $p=5$ in \cite{WanZho}.

For $\alpha \in \zz[C_p]$, we say that $\alpha \equiv 0 \pmod{n}$ for $n \in \nn$ if $\alpha \in n \zz[C_p]$. Applying Theorem~\ref{thm1}, we have the following result.

\begin{lem}
\label{lem11}
Let $u \in \mathcal{V}$ and $f_g, F_g$ be as above. We have
\begin{itemize}
\item[(i)]
$F_a^2 + F_b^2 + F_c^2 \equiv 0 \pmod{4}$ and
\item[(ii)]
$F_a + F_b + F_c \equiv 0 \pmod{2}.$
\end{itemize}
\end{lem}

\begin{proof}
(i) Since $F_g = f_g(t) + f_{g z}(t)$ and $\sum_{g \in \{a, b, c\}} (f_g(t) - f_{g z}(t))^2 = 0$ by Theorem~\ref{thm1}, the result follows from $$F_a^2 + F_b^2 + F_c^2 = 4 (f_a(t) f_{az}(t) + f_b(t) f_{bz}(t) + f_c(t) f_{cz}(t)).$$

(ii) For each $g \in \{a, b, c\}$, we write $$F_g = \sum_{i=0}^{p-1} g_i t^i \mbox{ for } g_i \in \zz.$$ For instance, $F_a = a_0 + a_1 t + \cdots + a_{p-1} t^{p-1}$ and $F_b, F_c$ are similar. Then $$F_g^2 = \left( \sum_{i=0}^{p-1} g_i t^i \right)^2 = \sum_{i=0}^{p-1} g_i^2 t^{2i} + 2 \sum_{i < j} g_i g_j t^{i+j} \equiv \sum_{i=0}^{p-1} g_i t^{2i} \pmod{2}$$ since $g_i^2 \equiv g_i \pmod{2}$. We have $$F_a^2 + F_b^2 + F_c^2 \equiv \sum_{i=0}^{p-1} (a_i + b_i + c_i) t^{2 i} \pmod{2}.$$ Note that we also have $F_a^2 + F_b^2 + F_c^2 \equiv 0 \pmod{2}$ by (i). It follows that $a_i + b_i + c_i \equiv 0 \pmod{2}$ for all $i$ since $2i$ runs over all $\{0, 1, \ldots, p-1\}$. Thus, $F_a + F_b + F_c \equiv 0 \pmod{2}$ and the result follows.
\end{proof}

We need a well-known result to prove the next proposition. The following lemma is a special case of \cite[Lemma~(2.10) (Cliff-Sehgal-Weiss)]{Seh2}.

\begin{lem}
\label{lem10}
When $p$ is an odd prime, we have $$\mathcal{U}_1(\zz[C_p]) = C_p \times \mathcal{U}_2(\zz[C_p]) \quad \text{and} \quad \mathcal{U}_2(\zz[C_p]) = \mathcal{U}_*(\zz[C_p])$$ where $$\mathcal{U}_2(\zz[C_p]) = \{v \in \mathcal{U}_1(\zz[C_p]) \mid v \equiv 1 \hspace{-0.5em} \mod (t-1)^2\}$$ and $$\mathcal{U}_*(\zz[C_p]) = \{v \in \mathcal{U}_1(\zz[C_p]) \mid v^* = v\}.$$
\end{lem}

\begin{prop}
\label{propFeb1}
Let $F_1, F_a, F_b, F_c$ and $w_1, w_a, w_b, w_c$ be as above.
\begin{itemize}
\item[(i)]
All units $w_1, w_a, w_b, w_c$ are in $\mathcal{U}_2(\zz[C_p])$ and so they are $*$-invariant.
\item[(ii)]
All elements $F_1, F_a, F_b, F_c$ are $*$-invariant. 
\end{itemize}
\end{prop}

\begin{proof}
(i) Note that $F_1 = f_1(t) + f_z(t) = 1 + 2 f_z(t)$ since $f_1(t) - f_z(t) = 1$. Moreover, $F_a + F_b + F_c \equiv 0 \pmod{2}$ by Lemma~\ref{lem11}(ii). It follows that $$w_g \equiv 1 \pmod{2}$$ for each $g \in \{1, a, b, c\}$ by equations (\ref{sec4eq1}). If we write $w_g = 1 + 2\alpha_g$ for $\alpha_g \in \zz[C_p]$, then we have $w_g^* = 1 + 2 \alpha_g^*$. Since $\alpha_g^* \in \zz[C_p]$, we can deduce that $w_g^* \equiv 1 \pmod{2}$. According to Lemma~\ref{lem10}, we have $\mathcal{U}_1(\zz[C_p]) = C_p \times \mathcal{U}_2(\zz[C_p])$. Since $w_g \in \mathcal{U}_1(\zz[C_p])$, there exists a unique $i_g \in \{0, 1, \ldots, p-1\}$ such that $t^{i_g} w_g \in \mathcal{U}_2(\zz[C_p]) = \mathcal{U}_*(\zz[C_p])$. We have $t^{i_g} w_g = (t^{i_g} w_g)^* = w_g^* t^{-i_g}$. Then we obtain $t^{i_g} \equiv t^{-i_g} \pmod{2}$. It follows that $i_g = 0$. Hence, $w_g \in \mathcal{U}_2(\zz[C_p]) = \mathcal{U}_*(\zz[C_p])$, as desired. 

(ii) We observe from (\ref{sec4eq1}) that \begin{align*} \left\{\begin{array}{ccl} 4 F_1 & = & w_1 + w_a + w_b + w_c, \\ 4 F_a & = & w_1 + w_a - w_b - w_c, \\ 4 F_b & = & w_1 - w_a + w_b - w_c, \\ 4 F_c & = & w_1 - w_a - w_b + w_c. \end{array} \right.\end{align*} Hence $F_g^* = F_g$ since the involution $*$ is linear and the characteristic of $\zz$ is zero. 
\end{proof}

Now we prove our main theorem.

\begin{thm}
\label{thm2sec4}
If the multiplicative order of $2$ modulo $p$ is congruent to $2$ modulo $4$, then $\zz[Q_8 \times C_p]$ has MJD.
\end{thm}


\begin{proof}
Let $u \in \mathcal{V}$, $u = \sum_{g \in Q_8} f_g(t) g$ and $F_g = f_g(t) + f_{gz}(t)$ as previous. To have the result, it suffices to show that $F_g \in 2 \zz[C_p]$ for $g \in \{a, b, c\}$ by Theorem~\ref{thm1}(i). First of all, we have $F_a + F_b + F_c \equiv 0 \pmod{2}$ by Lemma~\ref{lem11}(ii). Consider that $F_c \equiv F_a + F_b \pmod{2}$ and we can obtain $F_c^2 \equiv (F_a + F_b)^2 \pmod{4}$. It follows that $2 F_a^2 + 2 F_b^2 + 2 F_a F_b \equiv 0 \pmod{4}$ by Lemma~\ref{lem11}(i) so $$F_a^2 + F_b^2 + F_a F_b \equiv 0 \pmod{2}.$$ Multiplying $(F_a - F_b)$ on both sides, we have \begin{align*} F_a^3 \equiv F_b^3 \pmod{2}. 
\end{align*}

Write $F_g = \sum_{i=0}^{p-1} g_i t^i$ for $g_i \in \zz$ and $g \in \{a, b, c\}$. Note that we have the congruence $F_g^2 \equiv \sum_{i=0}^{p-1} g_i t^{2i} \pmod{2}$ (see the proof of Lemma~\ref{lem11}(ii)). Continue this squaring process and we can obtain that $$F_g^{2^k} \equiv \sum_{i=0}^{p-1} g_i t^{2^k i} \pmod{2} \quad \text{for any } k \in \nn.$$ Let $2 m$ be the multiplicative order of $2$ modulo $p$, then $2^{m} \equiv -1 \pmod{p}$. Moreover, $m$ is odd by assumption. We obtain that $t^{2^{m} i} = t^{-i} = (t^i)^*$ since $t^p = 1$. By Proposition~\ref{propFeb1}(ii), we have $$F_g^{2^{m}} \equiv \sum_{i=0}^{p-1} g_i  t^{2^m i}\equiv \sum_{i=0}^{p-1} g_i (t^{i})^* \equiv F_g^* \equiv F_g \pmod{2} \quad \text{for } g \in \{a, b, c\}.$$

Now, we observe that $2^m + 1 \equiv (-1)^m + 1 \equiv 0 \pmod{3}$ since $m$ is odd. Hence, by the congruence $F_a^3 \equiv F_b^3 \pmod{2}$, we obtain $$F_a^2 \equiv F_a^{2^m + 1} \equiv F_b^{2^m +1} \equiv F_b^2 \pmod{2}.$$ Therefore, $$F_a \equiv F_b \pmod{2}$$ by the congruence $F_g^2 \equiv \sum_{i=0}^{p-1} g_i t^{2i} \pmod{2}$ again. It follows that $F_c \equiv F_a + F_b \equiv 0 \pmod{2}$. By symmetry, we have $F_a \equiv F_b \equiv 0 \pmod{2}$.
\end{proof}

\begin{remk*}
In the previous proof, we basically show that if three $*$-invariant elements $\alpha, \beta, \gamma$ in $\zz[C_p]$ satisfy $\alpha^2 + \beta^2 + \gamma^2 \equiv 0 \pmod{4}$, then $\alpha \equiv \beta \equiv \gamma \equiv 0 \pmod{2}$ for certain primes $p$. However, it is not true for $p = 5$. For instance, if we consider that $\alpha = t + t^{4}$, $\beta = t^2 + t^{3}$ and $\gamma = \alpha + \beta$, then we still have $\alpha^2 + \beta^2 + \gamma^2 \equiv 0 \pmod{4}$ but $\alpha, \beta, \gamma \not\equiv 0 \pmod{2}$.
\end{remk*}

\section{A Proof for $p=5$}
\label{sec5}

In this section, we will show that $\zz[Q_8 \times C_5]$ has MJD and our proof is different from \cite{WanZho}. Let $p$ be an odd prime. For convenience, we write $\mathcal{U}_1 = \mathcal{U}_1(\zz[C_p])$, $\mathcal{U}_2 = \mathcal{U}_2(\zz[C_p])$ and $\mathcal{U}_* = \mathcal{U}_*(\zz[C_p])$. Recall that $C_p = \langle t \mid t^p = 1 \rangle$. Consider automorphisms $$\begin{array}{rccc} \varphi_i : & \zz[C_p] & \to & \zz[C_p] \\ & t & \mapsto & t^i \end{array}$$ for $1 \leq i \leq p-1$.  For each $i$, it is clear that $\varphi_i$ forms an automorphism of $\mathcal{U}_1$ since $\varphi_i$ preserves augmentation. Moreover, $\varphi_i(C_p) = C_p$ and $\varphi_i(\mathcal{U}_2) \subseteq \mathcal{U}_2$. It follows that $\varphi_i(\mathcal{U}_2) = \mathcal{U}_2$ because $\mathcal{U}_2 = \varphi_i (\varphi_j (\mathcal{U}_2) )\subseteq \varphi_i(\mathcal{U}_2)$ for some $\varphi_j$ with $\varphi_i \varphi_j = \id$, the identity map. By Lemma~\ref{lem10}, $\mathcal{U}_2 = \mathcal{U}_*$. Consequently, each $\varphi_i$ forms a group automorphism when restricting on $\mathcal{U}_1$, $\mathcal{U}_2$ and $\mathcal{U}_*$, respectively.

For $\alpha \in \zz[C_p]$, we define $$N(\alpha):= \prod_{i=1}^{p-1} \varphi_i(\alpha).$$ As in Section~\ref{sec2}, let $\varepsilon$ be a primitive $p$-th root of $1$ in $\cc$ and consider the ring homomorphism $$\delta: \zz[C_p] \to \zz[\varepsilon]$$ given by $t \mapsto \varepsilon$. For each $\alpha \in \zz[C_p]$, we have $$\delta(N(\alpha)) = N_{\qq(\varepsilon) / \qq} (\delta(\alpha))$$ where $N_{\qq(\varepsilon) / \qq}$ is the norm map on $\qq(\varepsilon)$. 

If $w \in \mathcal{U}_1$, then $N(w) \in \mathcal{U}_1$ and $\delta(N(w)) \in \mathcal{U}_1(\zz[\varepsilon])$ by Proposition~\ref{prop3}. On the other hand, $N_{\qq(\varepsilon) / \qq} (\delta(w)) = \pm 1$ since $\delta(w)$ is a unit in $\zz[\varepsilon]$. Thus, $\delta(N(w)) = \pm 1$. Note that $-1 \not\in \mathcal{U}_1(\zz[\varepsilon])$, otherwise $-1 \in 1 + (1 - \varepsilon) \zz[\varepsilon]$ and it shows $p \mid 2$ via $N_{\qq(\varepsilon) / \qq}(1 - \varepsilon) = p$. Hence, we have $\delta(N(w)) = 1$. Then, we get $N(w) = 1$ by Proposition~\ref{prop3} again. As a consequence, $$N(w) = 1 \quad \mbox{for} \quad w \in \mathcal{U}_1.$$

Observe that $U = \{\varphi_i \mid 1 \leq i \leq p-1\}$ forms a group isomorphic to $(\zz / p \zz)^{\times}$ via $\varphi_i \mapsto \overline{i}$ for each $i$. It follows that $\varphi_{p-1}$ is of order $2$ in $U$ and $$U = \bigcup_{i=1}^{(p-1)/2} \{\varphi_i\} \langle \varphi_{p-1} \rangle$$ is a union of coset representatives. We note that $\varphi_{p-1}(\alpha) = \alpha^*$ for any $\alpha \in \zz[C_p]$. Thus, if $w \in \mathcal{U}_2 = \mathcal{U}_*$, we have $\varphi_{p-1}(w) = w$ and it follows that $N(w) = v^2$ where $v = \prod_{i=1}^{(p-1)/2} \varphi_i (w)$. Since $\mathcal{U}_2 \subseteq \mathcal{U}_1$, we have $v^2 = 1$ from the above. We remark that if $H$ is a finite group and $v$ is a torsion unit of augmentation $1$ in $\zz[H]$, then the order of $v$ divides $|H|$ (see \cite[Lemma~37.3]{Seh2}). Since $|C_p| = p$ is an odd prime, we can conclude that $v = 1$. In other words, $$\prod_{i=1}^{(p-1)/2} \varphi_i (w) = 1 \quad \text{for} \quad w \in \mathcal{U}_2.$$ In particular, we have $w \cdot \varphi_2(w) = 1$ for $w \in \mathcal{U}_2$ when $p = 5$. Now, we can prove the following result.

\begin{thm}
\label{thm3}
$\zz[Q_8 \times C_5]$ has MJD.
\end{thm}

\begin{proof}
Let $u \in \mathcal{V}$, $u = \sum_{g \in Q_8} f_g(t) g$ and $F_g = f_g(t) + f_{gz}(t)$ as in Section~\ref{sec4}. As in the proof of Theorem~ \ref{thm2sec4}, we still have $F_a^3 \equiv F_b^3 \pmod{2}$ and it suffices to show that $F_a \equiv F_b \pmod{2}$. Recall the four units $w_1, w_a, w_b, w_c \in \mathcal{U}_2$ above Lemma~\ref{lem11}. Since $p = 5$, we have $$w_g^{-1} = \varphi_2(w_g)$$ for each $g \in \{1, a, b, c\}$. Now we consider $u^{-1}$ which is also non-semisimple and we can apply Theorem~\ref{thm1}. Write $u^{-1} = \sum_{g \in Q_8} h_g(t) g$ for $h_g(t) \in \zz[C_5]$ and $H_g = h_g(t) + h_{g z}(t)$ for $g \in \{1, a, b, c\}$. Recall the natural homomorphism $\phi : \qq[G] \to \qq[G/G']$ in Section~\ref{sec2} and we have $\phi(u^{-1}) = H_1 + H_a \overline{a} + H_b \overline{b} + H_c \overline{c}$. As (\ref{sec4eq1}) in Section~\ref{sec4}, we can obtain \begin{align*} \left\{\begin{array}{ccl} w_1^{-1} & = & H_1 + H_a + H_b + H_c, \\ w_a^{-1} & = & H_1 + H_a - H_b - H_c, \\ w_b^{-1} & = & H_1 - H_a + H_b - H_c, \\ w_c^{-1} & = & H_1 - H_a - H_b + H_c, \end{array} \right. \end{align*} since $\phi(u^{-1}) = \phi(u)^{-1}$. It follows that \begin{align*} \left\{\begin{array}{ccl} 4 H_1 & = & w_1^{-1} + w_a^{-1} + w_b^{-1} + w_c^{-1}, \\ 4 H_a & = & w_1^{-1} + w_a^{-1} - w_b^{-1} - w_c^{-1}, \\ 4 H_b & = & w_1^{-1} - w_a^{-1} + w_b^{-1} - w_c^{-1}, \\ 4 H_c & = & w_1^{-1} - w_a^{-1} - w_b^{-1} + w_c^{-1}. \end{array} \right.\end{align*} 
Now, we can deduce from $w_g^{-1} = \varphi_2(w_g)$ and from the equations in the proof of Proposition~\ref{propFeb1}(ii) that $$H_g = \varphi_2(F_g)$$ for each $g \in \{1, a, b, c\}$ since $\varphi_2$ is an automorphism of $\zz[C_5]$.

Now, by Lemma~\ref{lem7}, $u_s + u_s^{-1} \in \zz[G]$ which implies that $$\frac{1}{2} \sum_{g \in \{a, b, c\}} (F_g + H_g) g (1+z) \in \zz[G]$$ by Theorem~\ref{thm1}(i). Thus, we have $\frac{F_g + H_g}{2} \in \zz[C_p]$ for each $g \in \{a, b, c\}$. In other words, $F_g + H_g \equiv 0 \pmod{2}$. Then for each $g$, $$F_g \equiv H_g \equiv \varphi_2(F_g) \equiv \sum_{i=0}^{p-1} g_i t^{2i} \equiv F_g^2 \pmod{2}$$ if $F_g = \sum_{i=0}^{p-1} g_i t^{i}$. It follows that $F_g^2 \equiv F_g^3 \pmod{2}$. Because $F_a^3 \equiv F_b^3 \pmod{2}$, we obtain $$F_a \equiv F_a^2 \equiv F_a^3 \equiv F_b^3 \equiv F_b^2 \equiv F_b \pmod{2}.$$ Therefore, $\zz[Q_8 \times C_5]$ has MJD.
\end{proof}

\section{The Density of Primes}
\label{sec6}

In this section, we compute the Chebotarev density for primes satisfying the condition of Theorem~\ref{thm2}. For convenience, we denote $\ord_p(n)$ the multiplicative order of $n$ modulo $p$ and let $$P = \{\text{odd primes } p \mid \ord_p(2) \equiv 2 \pmod{4}\}.$$ To compute the Chebotarev density of $P$, it suffices to describe primes in $P$ by $\ord_p(2)$ and $\ord_p(4)$. Specifically, $$\text{$p \in P$ if and only if $\ord_p(2)$ is not odd and $\ord_p(4)$ is odd}$$ since $\ord_p(2) = 2 \cdot \ord_p(4)$ when $\ord_p(2)$ is even.

Recall the main result of \cite{Odo}.

\begin{thm}[{\cite[Theorem 1]{Odo}}]
\label{thmOdo}
Let $\mathscr{Q}$ be a finite non-empty set of primes $q$, with product $Q$, and let $g > 1$. Let $t \geq 1$ be the largest natural number such that $g$ is a $t$-th power in $\zz$ and, for each $q \in \mathscr{Q}$, let $q^{\tau(q)} \| t$. Suppose that $g = \hat{g}^t$, $\hat{g} \in \zz$, $\hat{g} > 1$, and that $\hat{g} = \tilde{g} \times \text{(square)}$ in $\zz$, where $\tilde{g}$ is squarefree. For each $q \in \mathscr{Q}$ let $q^{\gamma(q)} \| \tilde{g}$. Then, as $x \to \infty$, the number of primes $p \leq x$ such that no prime in $\mathscr{Q}$ divides $\ord_p(g)$ is asymptotically $$\lambda(\mathscr{Q}, g) {\rm Li}(x) + O \left( {\rm Li}(x) \exp\left( -c \frac{\log \log x}{\log \log \log x}\right)\right),$$ where ${\rm Li}(x) = \int_2^x d t / \log t$ and $$\lambda(\mathscr{Q}, g) = \prod_{q \in \mathscr{Q}} \left( 1 - \frac{q^{1 - \tau(q)}}{q^2 - 1} \right) + \lambda^*(\mathscr{Q}, g),$$ where $\lambda^*(\mathscr{Q}, g)$ is as follows:
\begin{itemize}
\item[(i)]
$\lambda^*(\mathscr{Q}, g) = 0$ if $2 \not\in \mathscr{Q}$;
\item[(ii)]
$\lambda^*(\mathscr{Q}, g) = 0$ if $\tilde{g} \nmid 2 Q$;
\item[(iii)]
if $2 \in \mathscr{Q}$ and $\tilde{g} \mid 2 Q$, then $$\lambda^*(\mathscr{Q}, g) = \prod_{q \in \mathscr{Q}} c_q(g),$$ where, for $q > 2$, $c_q(g) = 1 - \gamma(q) - (q^2 - 1)^{-1} q^{1 - \tau(q)}$, while, for $q = 2$, $$c_2(g) = \left\{ \begin{array}{ll} (2^{\tau(2)} 3)^{-1} & \text{if } \tilde{g} \equiv 1 \pmod{4} \\ (2^{\tau(2)} 3)^{-1} - \sum_{1 + \tau(q) \leq n < 2} \frac{1}{2} & \text{if } \tilde{g} \equiv 3 \pmod{4} \\ (2^{\tau(2)} 3)^{-1} & \text{if } \tilde{g} \equiv 2 \pmod{4} \text{ and } 4 \mid t \\ -\frac{1}{12} & \text{if } \tilde{g} \equiv 2 \pmod{4} \text{ and } 2 \, \| \, t \\ - \frac{1}{24} & \text{if } \tilde{g} \equiv 2 \pmod{4} \text{ and } 2 \nmid t. \end{array} \right.$$
\end{itemize}
\end{thm}

The constant $\lambda(\mathscr{Q},g)$ is the Chebotarev density for those primes $p$ that $\ord_p(g)$ can not be divided by any prime in $\mathscr{Q}$. For instance, we can recover Hasse's result.

\begin{cor}[{\cite[Section 3]{Has}}]
\label{corHas}
The density of primes $p$ such that $\ord_p(2)$ is odd is $7/24$. Therefore, the density of primes $p$ such that $\ord_p(2)$ is even is $17 / 24$.
\end{cor}

\begin{proof}
To compute the density of primes $p$ such that the multiplicative order of $2$ modulo $p$ is odd, we can choose $\mathscr{Q} = \{2\}$ and $g = 2$. Then $t = 1$, $\tau(q) = 0$, $\tilde{g} = g = 2$, and $$\lambda(\mathscr{Q}, g) = \left(1 - \frac{2^{1-0}}{2^2 - 1} \right) + \lambda^*(\mathscr{Q}, g) = \frac{1}{3} - \frac{1}{24} = \frac{7}{24}.$$
\end{proof}

Using the previous corollary, we can compute the density of $P$.

\begin{thm}
\label{thm1sec6}
The density of $P$ is $7 / 24$ and then $P$ is an infinite set.
\end{thm}

\begin{proof}
To count primes in $P$, it is equivalent to count primes $p$ such that $\ord_p(4)$ is odd and $\ord_p(2)$ is not odd. First of all, if $\ord_p(2)$ is odd, then $\ord_p(4)$ is also odd. Moreover, the density of primes $p$ which $\ord_p(4)$ is odd is given by $\lambda(\{2\}, 4)$ and the density of primes $p$ which $\ord_p(2)$ is odd is given by $\lambda(\{2\},2)$. Therefore, by Theorem~\ref{thmOdo} and Corollary~\ref{corHas}, the density of $P$ is $$\lambda(\{2\}, 4) - \lambda(\{2\}, 2) = \left(1 - \frac{2^{1-1}}{2^2 - 1} - \frac{1}{12} \right) - \frac{7}{24} = \frac{7}{24}.$$ 
\end{proof}

We remark here that there are 2917 primes in $P$ for the first 10000 primes.

As a consequence, $\zz[Q_8 \times C_p]$ has MJD for infinitely many primes $p$ with even $\ord_p(2)$.

\section*{Acknowledgment}

The material in this paper is part of the second author's Ph.D. thesis written at National Taiwan Normal University. This work was done when the first author visited Taida Institute for Mathematical Sciences in Taiwan in 2018. The first author would like to thank the invitation and hospitality of the institute. The research of the first author was partially supported by an NSERC discovery grant. The second author would like to thank his advisor, Professor Chia-Hsin Liu, for helpful discussions and continuous encouragement. The authors would like to thank the referee for kind and careful comments.

\bibliographystyle{alpha} 
\bibliography{Bib_MJD} 

\begin{thebibliography}{HPW12}

\bibitem[AHP98]{AroHalPas}
S.R. Arora, A.W. Hales, and I.B.S. Passi.
\newblock The multiplicative {J}ordan decomposition in group rings.
\newblock {\em J. Algebra}, 209:533--542, 1998.

\bibitem[GS95]{GiaSeh}
A.~Giambruno and S.K. Sehgal.
\newblock Generators of large subgroups of units of integral group rings of
  nilpotent groups.
\newblock {\em J. Algebra}, 174:150--156, 1995.

\bibitem[Has66]{Has}
H.~Hasse.
\newblock {\"{U}ber die Dichte der Primzahlen $p$ f\"{u}r die eine vorgegebebe
  ganzrationale Zahl $a \neq 0$ von gerader bzw. ungerader Ordnung mod. $p$
  ist.}
\newblock {\em Math. Ann.}, 166:19--23, 1966.

\bibitem[HP91]{HalPas}
A.W. Hales and I.B.S. Passi.
\newblock Integral group rings with {J}ordan decomposition.
\newblock {\em Arch. Math.}, 57:21--27, 1991.

\bibitem[HP17]{HalPas2}
A.W. Hales and I.B.S. Passi.
\newblock Group rings and {J}ordan decomposition.
\newblock In {\em Groups, rings, group rings, and Hopf algebras}, volume 688 of
  {\em Contemp. Math.}, pages 103--111, Providence, RI, 2017. Amer. Math. Soc.

\bibitem[HPW07]{HalPasWil}
A.W. Hales, I.B.S. Passi, and L.E. Wilson.
\newblock {The multiplicative {J}ordan decomposition in group rings, II}.
\newblock {\em J. Algebra}, 316:109--132, 2007.

\bibitem[HPW12]{HalPasWil2}
A.W. Hales, I.B.S. Passi, and L.E. Wilson.
\newblock {Corrigendum to ``The multiplicative Jordan decomposition in group
  rings, II'' [J. Algebra 316 (1) (2007) 109-132]}.
\newblock {\em J. Algebra}, 371:665--666, 2012.

\bibitem[LP09]{LiuPas}
C.-H. Liu and D.S. Passman.
\newblock {Multiplicative Jordan decomposition in group rings of $3$-groups}.
\newblock {\em J. Algebra Appl.}, 8(4):505--519, 2009.

\bibitem[LP10]{LiuPas2}
C.-H. Liu and D.S. Passman.
\newblock {Multiplicative Jordan decomposition in group rings of $2,3$-groups}.
\newblock {\em J. Algebra Appl.}, 9(3):483--492, 2010.

\bibitem[LP13]{LiuPas3}
C.-H. Liu and D.S. Passman.
\newblock {Multiplicative Jordan decomposition in group rings with a Wedderburn
  component of degree $3$}.
\newblock {\em J. Algebra}, 388:203--218, 2013.

\bibitem[LP14]{LiuPas4}
C.-H. Liu and D.S. Passman.
\newblock {Multiplicative Jordan decomposition in group rings of $3$-groups,
  II}.
\newblock {\em Comm. Algebra}, 42(6):2633--2639, 2014.

\bibitem[Odo81]{Odo}
R.W.K. Odoni.
\newblock {A conjecture of Krishnamurthy on decimal periods and some allied
  problems}.
\newblock {\em J. Number Theory}, 13:303--319, 1981.

\bibitem[Par02]{Par}
M.M. Parmenter.
\newblock Multiplicative {J}ordan decomposition in integral group rings of
  groups of order $16$.
\newblock {\em Comm. Algebra}, 30(10):4789--4797, 2002.

\bibitem[PS02]{PMSeh}
C.~{Polcino Milies} and S.K. Sehgal.
\newblock {\em An Introduction to Group Rings}, volume~1 of {\em Algebras and
  Applications}.
\newblock Kluwer Academic Publishers, Dordrecht, 2002.

\bibitem[Seh78]{Seh}
S.K. Sehgal.
\newblock {\em Topics in Group Rings}, volume~50 of {\em Monographs and
  Textbooks in Pure and Applied Math.}
\newblock Marcel Dekker, Inc., New York, 1978.

\bibitem[Seh93]{Seh2}
S.K. Sehgal.
\newblock {\em Units in {I}ntegral Group {R}ings}, volume~69 of {\em Pitman
  Monographs and Surveys in Pure and Applied Math.}
\newblock Longman Scientific \& Technical, Harlow; copublished in the United
  States with John Wiley \& Sons, Inc., New York, 1993.
\newblock With an appendix by Al Weiss.

\bibitem[WZ17]{WanZho}
{X.-L.} Wang and {Q.-X.} Zhou.
\newblock {Multiplicative Jordan decomposition in integral group ring of group
  $K_8 \times C_5$}.
\newblock {\em Commun. Math. Res.}, 33(1):64--72, 2017.

\end{thebibliography}



\end{document}